\documentclass[11pt]{article}
\usepackage{amsmath, amsfonts, amssymb}
\usepackage{color}
\usepackage{cite}
\usepackage{graphicx, amsthm}
\newcommand{\nc}{\newcommand}
\nc{\R}{{\mathbb{R}}}
\nc{\dist}{\mbox{\rm dist}\,}
\nc{\gph}{\mbox{\rm gph}\,}
\nc{\rge}{\mbox{\rm rge}\,}
\nc{\cl}{\mbox{\rm cl}\,}
\nc{\ri}{\mbox{\rm ri}\,}
\nc{\bd}{\mbox{\rm bd}\,}
\nc{\inter}{\mbox{\rm int}\,}
\nc{\diam}{\mbox{\rm diam}\,}
\nc{\proj}{\mbox{\rm proj}}
\nc{\reg}{\mbox{\rm reg}\,}
\nc{\lip}{\mbox{\rm lip}\,}
\nc{\dom}{\mbox{\rm dom}\,}
\nc{\epi}{\mbox{\rm epi}\,}
\nc{\mesh}{\mbox{\rm mesh}\,}

\nc{\conv}{\mbox{\rm conv}\,}
\nc{\X}{\mathbb{R}^n}
\nc{\Y}{\mathbb{R}^m}
\nc{\lF}{\textit{\tiny F}}
\nc{\lD}{\textit{\tiny D}}
\nc{\ldom}{\mbox{\rm{\tiny dom}}\,}

\newtheorem{theorem}{Theorem}[section]
\newtheorem{lemma}[theorem]{Lemma}
\newtheorem{claim}[theorem]{Claim}
\newtheorem{corollary}[theorem]{Corollary}

\theoremstyle{definition}
\newtheorem{definition}[theorem]{Definition}

\theoremstyle{remark}
\newtheorem{remark}[theorem]{Remark}

\numberwithin{equation}{section}

\begin{document}

\title{Orbits of geometric descent}


\author{
A. Daniilidis\thanks{
    DIM-CMM, Universidad de Chile, Blanco Encalada~2120,
    piso~5, Santiago, Chile;
    {\tt http://www.dim.uchile.cl/{\raise.17ex\hbox{$\scriptstyle\sim$}}arisd/}.
    Research supported by the grant MTM2011-29064-C01 (Spain) and FONDECYT Regular No 1130176 (Chile).
    }
    \and
D. Drusvyatskiy\thanks{%
    Department of Combinatorics and Optimization,
    University of Waterloo,
    Ontario N2L 3G1, Canada;
    Department of Mathematics,
    University of Washington,
    Seattle, WA 98195;
    {\tt http://people.orie.cornell.edu/dd379/}.
    }%
    \and
  A.S. Lewis\thanks{%
  School of Operations Research and Information Engineering,
  Cornell University,
  Ithaca, New York, USA;
  {\tt http://people.orie.cornell.edu/aslewis/}.
  Research supported in part by National Science Foundation Grant DMS-1208338.
}}


\date{}




\maketitle
\begin{abstract}
We prove that quasiconvex functions always admit descent trajectories bypassing all non-minimizing critical points.  
\end{abstract}

\section{Introduction}

To motivate the discussion, consider the classical gradient dynamical system 
\begin{equation}\label{eqn:grad_desc}
\dot{x}=-\nabla f(x), \quad\textrm{ where } f \textrm{ is a } C^1\textrm{-smooth function on }\R^d. 
\end{equation}
This differential equation always admits solutions starting from any point $x_0$, while
uniqueness is only assured when the gradient $\nabla f$ is Lipschitz continuous. In
this case, maximal trajectories of the system
never encounter a singularity of $f$ --- a point where the gradient $\nabla f$ vanishes --- in finite time. Instead, bounded trajectories converge
in the limit to the critical set of the function. True convergence to a limit point is a more delicate matter; it is only guaranteed under extra assumptions on the function $f$, such as convexity \cite{bruck75,DDDL-preprint} or analyticity \cite{Kurdyka98,BDLM-2010} for example.

Reparametrizing the orbits of (\ref{eqn:grad_desc}) by arclengths, at least away from singularities, we may instead seek absolutely continuous curves $x\colon [0,\eta)\to\R^d$ satisfying
\begin{equation}\label{eqn:grad_desc_param}
\dot{x}=-\frac{\nabla f(x)}{\|\nabla f(x)\|}, \qquad\textrm{ for a.e. } t\in [0,\eta),  
\end{equation}
where $\|\cdot\|$ denotes the norm on $\R^d$ and we temporarily adopt the convention $\frac{0}{0}=0$.
In comparison with (\ref{eqn:grad_desc}), this system is much more intrinsic to the geometry of the level sets of $f$.
Indeed, whenever $\nabla f$ is nonzero at a point $x$, the level set $[f=f(x)]$ is a smooth hypersurface around $x$ and the right hand side of (\ref{eqn:grad_desc_param}) coincides (up to sign) with the unit normal $\hat{n}(x)$
to the level set $[f=f(x)]$  at $x$. Consequently the orbits of the system (\ref{eqn:grad_desc_param}) may reach a
singularity in finite time and continue from there onward while not stopping at inessential singularities --- points $x$ where the gradient $\nabla
f(x)$ vanishes but the level set $[f=f(x)]$ is a hypersurface around $x$. To emphasize this distinction further, observe that the range of any smooth function can clearly be reparametrized to force a singularity at any prespecified point; on the other hand, such a reparametrization does not effect the level set portrait of the function.


A particularly important situation arises when the function $f$ is quasiconvex --- meaning its sublevel sets $[f\leq r]$ are convex. Such functions play a decisive role for example in the theory of utility functions in microeconomics; see the landmark paper \cite{AD54}.
In this case, we may even drop the smoothness assumption on $f$ and instead  seek, in analogy to (\ref{eqn:grad_desc_param}), absolutely continuous curves $x\colon [0,\eta)\to\R^d$ satisfying the inclusion 
\begin{equation}\label{eqn:diff_inc}
\dot{x}\in - N_{[f\leq f(x)]}(x), \quad\textrm{ for a.e. } t\in [0,\eta),
\end{equation}
where $N_{[f\leq f(x)]}(x)$ denotes the convex normal cone to the sublevel set. In this short note, we prove that this system (under very mild assumptions on $f$) always admits Lipschitz continuous trajectories starting from any point. Moreover maximally defined trajectories are either unbounded or converge to the global minimum of the function.\footnote{While completing this short note, we became aware of the preprint \cite{italiens}, where the authors address questions of a similar flavor.}

We should note a similarity of the differential inclusion (\ref{eqn:diff_inc}) to the classical Moreau's Sweeping process introduced in \cite{M77}; for a nice expository article see \cite{KM2000}. The standard assumption for the sweeping process to admit a solution (within an appropriate space of curves) is for the sweeping set mapping to be continuous and of bounded variation. Then one can reparametrize the problem so that the sweeping set mapping becomes Lipschitz continuous and then apply the standard ``catching up algorithm''; see \cite{KM2000} for details. In contrast, in the setting of the current manuscript the sublevel set mapping $t\mapsto [f\leq t]$ is not guaranteed to have bounded variation (see \cite[Section 4.3]{BDLM-2010} for a counter-example). Instead, the fundamental observation driving our analysis is that the polygonal curves created by the ``catching up algorithm'' are automatically self-contracted (Definition~\ref{Definition_Self-contracted}) and hence have finite length whenever they are bounded \cite{DDDL-preprint}, \cite[Theorem~3.3]{MP-1991}. This insight allows us to switch to the length parametrization and then apply the standard machinery of the theory of differential inclusions.

\section{Trajectories of convex foliations}
Throughout, we denote by $\R^d$ the $d$-dimensional Euclidean space. The corresponding inner-product and norm will be denoted by $\langle \cdot, \cdot \rangle$ and $\|\cdot\|$ respectively. For any subset $Q$ of $\R^d$, the symbols $\mathrm{int}\,Q$, $\partial Q$, and $\cl Q$ will denote the topological interior, boundary, and closure of $Q$, respectively. 
The {\em distance} of a point $x$ to $Q$ is   $$d(x,Q):=\inf_{y\in Q} d(x,y),$$ and the {\em metric projection} of $x$ onto $Q$ is $$P_Q(x):=\{y\in Q: d(x,y)=d(x,Q)\}.$$
Given points $x,y\in{\mathbb{R}}^{d}$ we define the closed
segment
\[
\lbrack x,y]:=\{tx+(1-t)y:t\in\lbrack0,1]\}.
\]
A subset $Q$ of ${\mathbb{R}}^{d}$ is {\em convex} if for every pair of points $x,y\in C$ the line segment $[x,y]$ lies in $Q$. The {\em convex hull} of any set $Q\subset\R^d$, namely the intersection of all convex sets containing $Q$, will be denoted by $\conv Q$.

The following notion, introduced in \cite[Section 6.3]{DLS-2010} and
further studied in \cite[Section 4.1]{DDDL-preprint}, is the focus of this short note.

\begin{definition}[Convex foliation]\label{Definition_foliation}
An ordered family of sets $\{S_t\}_{t\in [a,b]}$, indexed by an interval $[a,b]\subset\R$, is called a {\em convex foliation} provided the following properties hold. 
\begin{enumerate}
\item The sets $S_t$ are nonempty, closed, convex subsets of $\R^d$.
\item 
The implication 
\begin{equation*}
t_1 < t_2 \quad \Longrightarrow \quad S_{t_1}\subset\mathrm{int}\, S_{t_2} \quad\quad \textrm{ holds}.
\end{equation*}
\item 
The equation
\begin{equation*}
\bigcup_{t\in [a,b]}
\partial S_t=S_b\setminus(\mathrm{int}\,S_a)\, \quad\quad \textrm{holds}.
\end{equation*}
\end{enumerate}
For each point $x\in S_b\setminus(\mathrm{int}\,S_a)$, abusing notation slightly, we define the set $S_x$ to be the unique set of the convex foliation satisfying $x\in\partial S_x$. 
\end{definition}

\begin{remark}
We mention in passing that any convex foliation can be represented in terms of sublevel sets of an lsc quasiconvex function $f\colon \R^d\to \R\cup\{+\infty\}$ that is continuous on its domain and has no nonglobal extrema; conversely, sublevel sets of any such function naturally define a convex foliation.
\end{remark}

For any convex subset $Q$ of ${\mathbb{R}}^{d}$ and any point
$\bar{x}\in Q$ the {\em normal cone} $N_Q(\bar{x})$ has the classical description:
\[
N_{Q}(\bar{x})=\left\{  v\in{\mathbb{R}}^{d}:\ \langle v,x-\bar{x}\rangle\leq0,\ \textrm{ for all }
x\in Q\right\}  .
\]
The following is a key definition of the current work.
\begin{definition}[Trajectories of convex foliations]
A curve $\gamma$ is a {\em trajectory of a convex foliation} $\{S_t\}_{t\in [a,b]}$ if it admits an absolutely continuous parametrization $\gamma\colon I\to\R^d$
satisfying 
\begin{equation*}
\dot{\gamma}(\tau)\in - N_{S_{\gamma(\tau)}} (\gamma(\tau)) \quad \textrm{ for almost every }\tau\in I,
\end{equation*}
and for any $\tau_1,\tau_2\in I$ with $\tau_1 < \tau_2$  we have
$\gamma(\tau_2)\subset \inter S_{\gamma(\tau_1)}$.
\end{definition}
 
Our goal in this short note is to prove that trajectories of convex foliations always exist.
The following notion turns out to be instrumental. For more details see \cite{DLS-2010}.

\begin{definition}
[Self-contracted curve]\label{Definition_Self-contracted} 
A curve
$\gamma\colon I\rightarrow\mathbb{R}^{d}$ is
called \emph{self-contracted} if for any $t^{*}\in I$, the mapping 
$$t\mapsto d(\gamma(t),\gamma(t^{*})),\quad \textrm{ is nonincreasing on } I\cap (-\infty,t^{*}].$$
\end{definition}

The following result concerning lengths of self-contracted curves will be key for us. See \cite{MP-1991} for
Lipschitz curves and \cite[Theorem~3.3]{DDDL-preprint} for general (possibly
discontinuous) self-contracted curves.
\begin{lemma}[Lengths of self-contracted curves]\label{lem:len_bound}
Consider a self-contracted curve $\gamma\colon I\to\R^d$ and let $\Gamma\subset\R^d$ be the image of $I$ under $\gamma$. Then we have the estimate
$${\rm length}(\gamma)\leq K_d\,\diam(\Gamma),$$
where $K_d$ is a constant that depends only on the dimension $d$. 
\end{lemma}

We arrive at the main result of this short note.
\begin{theorem}
[Trajectories of convex foliations exist]\label{Theorem_main}
Consider a convex foliation $\{S_t\}_{t\in [a,b]}$. Then for any point $x_0\in S_b$ there exists a self-contracted curve $\gamma\colon [0,L]\to\R^d$ that is a
trajectory of the convex foliation and satisfies $\gamma(0)=x_0$ and $\gamma(L)\in S_a$.
\end{theorem}
\noindent\textbf{Proof.} 
Before we begin, we record the following result which will be used in the sequel. The proof is based
on a standard convexity argument and will be omitted. We defer to \cite[Definition~5.4]{RW98} for the relevant definitions of continuity of set-values mappings.
\begin{claim}
\label{Claim_2.3}If $\{S_t\}_{t\in\lbrack a,b]}$ is a convex foliation, then
the mappings $t\mapsto S_t$ and $x\mapsto N_{S_x}(x)$ are continuous in a
set-valued sense.
\end{claim}

Consider a partition $a=\tau_n < \tau_{n-1} < \ldots < \tau_{1} < \tau_0=b$ of the interval $[a,b]$.
Now inductively define the points 
\begin{equation}
x_i=\mathrm{proj}_{S_{\tau_i}}(x_{i-1})\quad \textrm{ for } i=1,\ldots, n.
\label{proj}
\end{equation} 
and consider the polygonal line
$$\Gamma_n= \bigcup_{i=0}^{n-1}
[x_{i},x_{i+1}].$$
Let $\gamma_n\colon [0,L_n]\to\R^d$ be the arclength parametrization of $\Gamma_n$. The following is true.
\begin{claim}
The curves $\gamma_n$ are self-contracted and satisfy $L_n\leq K_d\,\dist(x_0,S_a),$ where $K_d$ is a constant depending only on the dimension $d$.
\end{claim}
\begin{proof}
Fix an index $i\in\{0,\ldots,n-1\}$. Since $S_{\tau_{i+1}}$ is convex and we
have $x_{i}-x_{i+1}\in N_{S_{\tau_{i+1}}}(x_{i+1})$, it follows that for every
fixed $x\in S_{\tau_{i+1}}$, the function
\[
\theta\mapsto\Vert x_{i+1}+\theta(x_{i}-x_{i+1})-x\Vert
,\quad\theta\geq0,
\]
is non-decreasing. In particular, for any point $x\in S_a$ we have
\[
\Vert x_{i}-x\Vert\geq\Vert x_{i+1}-x\Vert.
\]
Since $i$ was arbitrary, we deduce $\mathrm{dist}(x_{0},S_a)\geq\Vert
x_{i+1}-\proj_{S_a}(x_{0})\Vert$ and consequently all the curves $\gamma_{n}$
are contained in a ball of radius $\mathrm{dist}(x_{0},S_a)$ around
$\proj_{S_a}(x_{0})$.

Consider now real numbers $0\leq e <f <g \leq L$. In the case that $\gamma
(e)$, $\gamma(f)$, $\gamma(g)$ all lie in a single line segment $[x_{i},x_{i+1}]$, the inequality
\[
\|\gamma(g)-\gamma(f)\|\leq\| \gamma(g)-\gamma(e)\|,
\]
is obvious. Hence we may suppose that there are indices $0\leq i_{1}\leq
i_{2}\leq i_{3}\leq n$, that are not all the same, and satisfying
\[
\gamma(e)\in[x_{i_{1}},x_{{i_{1}}+1}],\quad\quad\gamma(f)\in[x_{i_{2}},x_{{i_{2}}+1}],\quad\quad\gamma(g)\in[x_{i_{3}},x_{{i_{3}}+1}].
\]
Observe that the inclusion
\[
\gamma(g)\in S_{\tau_{i}}\text{,\quad\quad holds whenever }i_{1}\leq
i<i_{2},
\]
Consequently for such indices $i$, we have
\[
\Vert x_{i}-\gamma(g)\Vert\leq\Vert x_{i_{1}}-\gamma(g)\Vert.
\]

It follows immediately that the polygonal curve $\gamma$ is self-contracted.
The bound on the length of $\Gamma_{n}$ now follows directly from
Lemma~\ref{lem:len_bound}.
\end{proof}

In light of the claim above, the lengths of the curves $\gamma_n$ are bounded by a uniform constant 
$$L_{*}:=K_d\,\dist(x_0, S_a).$$ 
We can thus extend the domains of the curves $\gamma_n$ from
$[0,L_n]$ to $[0,L_{\ast}]$ (and continue to denote by
$\gamma_n$ the new curves for simplicity) as follows:
\[
\gamma_n(s)=\gamma_n(L),\quad\text{for every }s\in\lbrack
L,L_{\ast}].
\]

Now let the mesh of the partition $a=\tau_n < \tau_{n-1} < \ldots < \tau_{1} < \tau_0=b$ tend to zero as $n$ tends to $\infty$.
Clearly
each curve $\gamma_{n}$ is
$1$-Lipschitz. It follows that the sequence
$\{\gamma_{n}\}_{n}$ is
equi-continuous and equi-bounded, and hence by the Arzela-Ascoli theorem
(see for example \cite[Section 7]{K75}) it has a subsequence, which we still denote by $\{\gamma_{n}\}_{n}$, that converges uniformly to a curve $\gamma\colon [0,L_{*}]\to\R^d$.
It follows that $\gamma$ is a self-contracted, $1$-Lipschitz continuous
curve, satisfying $\gamma(0)=x_{0}.$ In particular the inequality $\|\dot{\gamma}(s)\|\leq 1$ holds almost everywhere on $[0,L_{*}]$.
Consider now the sequence of derivatives
$\{\dot{\gamma}_{n}\}_{n}$ in the Hilbert space $L^{2}([0,L_{\ast
}],{\mathbb{R}}^{d})$ (equipped with the $\|\cdot\|_{2}$-norm).
Notice that the inequalities
$\|\dot{\gamma}_{n}\|_{2}\leq\sqrt{L_{\ast}}$ hold for all
$n$. Thus the sequence $\{\dot{\gamma}_{n}\}_{n}$ has a weakly converging
subsequence, which we still denote by $\{\dot{\gamma}_{n}\}_{n}$. A standard argument easily shows that this limit coincides with $\dot{\gamma}$ almost everywhere on $[0,L_{\ast }]$.

Mazur's Lemma then implies that a subsequence of convex combinations of the form $\sum^{K(n)}_{k=n} \alpha^n_k \dot{\gamma}_k$ converges strongly to $\dot{\gamma}$ as $n$ tends to $\infty$. Since convergence in $L^2[0,L_{*}]$ implies almost everywhere pointwise convergence, we deduce that for almost every $s\in [0,L_{*}]$, we have
\[
\Big\|\sum_{k=n}^{K(n)}\alpha_{k}^{n}\,\dot{\gamma}_{k}(s)-\dot{\gamma
}(s)\Big\|\rightarrow0,\quad\text{as }n\rightarrow\infty.
\]
Fix such a number $s\in\lbrack0,L_{\ast}]$. Then by Carath\'{e}odory's theorem
we may assume that the quantity $K(n)-(n-1)$ is bounded by $d+1$. Relabelling
we then have
\[
\lim_{n\rightarrow\infty}\sum_{i=1}^{d+1}\lambda_{i}^{n}\,\dot{\gamma}_{{i}
}^{n}(s)=\dot{\gamma}(s).
\]
Passing successively to subsequences, we may assume that
\begin{equation}
\dot{\gamma}_{{i}}^{n}(s)\rightarrow v_{i}(s)\text{,\quad for all }
i\in\{1,\ldots,d+1\}\text{,} \label{g21}%
\end{equation}
and similarly,
\[
(\lambda_{1}^{n},\ldots,\lambda_{d+1}^{n})\rightarrow(\lambda_{1}
,\ldots,\lambda_{d+1}).
\]
Consequently we obtain the inclusion
\begin{equation}
\dot{\gamma}(s)\in\conv\{v_{1},\ldots,v_{d+1}\}.
\end{equation}
By construction for each $i\in\{1,\ldots,d+1\}$ and $n\in\mathbb{N}$, there
exist real numbers  $\tau_{i_{n}}^{-}>\tau_{i_{n}}^{+}$ and corresponding
$s_{{i}_{n}}^{-}<s_{{i}_{n}}^{+}$ satisfying $S_{\gamma_{i_{n}}
(s_{{i}_{n}}^{-})}=S_{\tau_{i_{n}}^{-}}$ and $S_{\gamma_{i_{n}}(s_{{i}
_{n}}^{+})}=S_{\tau_{i_{n}}^{+}}$ and so that%
\[
\qquad\gamma_{i_{n}}(s)\in\lbrack\gamma_{i_{n}}(s_{{i}_{n}}^{-}),\gamma
_{i_{n}}(s_{{i}_{n}}^{+})],\qquad\qquad\dot{\gamma}_{i_{n}}(s)\in
-N_{S_{\tau_{{i}_{n}}^{+}}}(\gamma_{{i_{n}}}(s_{i_{n}}^{+}))
\]
Now observe $\Vert\gamma_{i_{n}}(s_{{i}_{n}}^{-})-\gamma_{i_{n}}(s_{{i}_{n}}^{+})\Vert=d(\gamma_{i_{n}}(s_{{i}_{n}}^{-}),S_{\tau_{i_{n}}^{+}})$.
According to Claim~\ref{Claim_2.3} the set-valued mapping $t\mapsto
S_t$ is continuous, whence we obtain $\Vert\gamma_{i_{n}}(s_{{i}_{n}}
^{-})-\gamma_{i_{n}}(s_{{i}_{n}}^{+})\Vert\rightarrow0$. The outer
semicontinuity of the mapping $x\mapsto N_{S_x}(x)$ (Claim~\ref{Claim_2.3}), along with \eqref{g21} immediately yields
\begin{equation}
-\dot{\gamma}(s)\in N_{S_{\gamma(s)}}(\gamma(s)),\quad\text{ for a.e. }
s\in\lbrack0,L_{\ast}]. \label{eqn:main_res}
\end{equation}

Let $L$ be the total length of the self-contracted curve $\gamma$. We now reparametrize $\gamma$ by arc-length and continue to
denote the resulting curve by $\gamma$ (since no confusion will
arise). This curve is now defined on $[0,L]$ and satisfies equation \eqref{eqn:main_res} with $\|\dot{\gamma}(s)\|=1,$ a.e.

Now to complete the proof, assume towards a
contradiction, that for some $s_{1}<s_{2}$ and all $s\in\lbrack
s_{1},s_{2}]$ the set $S_{\gamma(s)}$ is constantly equal to some set $Q$. Then we have
$(\delta_{Q}\circ\gamma)(s)=0,$ for all $s\in\lbrack s_{1},s_{2}].$ Then
by \cite[Theorem 10.6]{RW98} we have for almost all $s$ and all
$v(s)\in N_{Q}(\gamma(s))$
\[
\frac{d}{dt}(\delta_{Q}\circ\gamma)(s)=\langle\dot{\gamma}(s),v(s)\rangle=0.
\]
In view of \eqref{eqn:main_res} this yields $\|\dot{\gamma}(s)\|=0$ a.e. on
$[s_{1},s_{2}].$ This
contradicts the fact that $\gamma$ is parametrized by arclength, and concludes the proof. $\hfill\Box$

\begin{corollary}
[Smooth convex foliations]\label{Corollary_smooth bd} Consider a convex foliation $\{S_t\}_{t\in [a,b]}$ and suppose moreover that the sets $\partial S_t$ are $C^1$-smooth manifolds for each $t\in [a,b]$. Then
every trajectory $\gamma\colon I\to\R^d$ of the convex foliation can be parametrized by arclength, at which point it becomes $C^{1}$-smooth on the interior of its domain of definition.
\end{corollary}
\noindent\textbf{Proof.} Observe that for every point $x\in S_b\setminus \inter S_a$, there exists a unitary normal vector $\hat
{n}(x)\in\mathbb{R}^{d}$ satisfying
\[
N_{S_x}(x)=\mathbb{R}_{+}\hat{n}(x).
\]
The assignment $x\mapsto\hat
{n}(x)$  is a unitary continuous vector field on $ S_b\setminus \inter S_a$. On the other hand, when  $\gamma$  is parametrized by arclength, we have $\dot{\gamma }(s)=\hat{n}(\gamma(s))$ a.e. on $\gamma$'s domain of definition. Since we have the representation 
\[
\gamma(s)=\gamma(0)+ \int_{0}^{s}
\dot{\gamma}(\tau) \; d\tau=\gamma(0)+\int_{0}^{s}
\hat{n}(\gamma(\tau)) \; d\tau,
\]
we deduce that $\gamma$ is a $C^{1}$-smooth curve on the interior of its domain.$\hfill\Box$




\bigskip
\textbf{Acknowledgements: }The first author thanks David Marin (UAB) for useful discussions.
Part of this work has been realized during a research stay of the first author
at Cornell University (December 2012). This author thanks his hosts and
the host institution for hospitality.


\bibliographystyle{amsplain}

\end{document}